\documentclass[letterpaper]{ieeeconf}

\IEEEoverridecommandlockouts  

\usepackage{cite}

\usepackage[pdftex]{graphicx}
\graphicspath{{./images/}}
\usepackage{dblfloatfix}
\usepackage[cmex10]{amsmath}
\usepackage{labtex}
\usepackage[tight,footnotesize]{subfigure}
\usepackage{float}
\usepackage{color}
\usepackage{hhline}

\usepackage{amsthm} 

\newtheorem{thm}{\protect\theoremname}
\newtheorem{defn}[thm]{\protect\definitionname}
\newtheorem{prop}[thm]{\protect\propositionname}
\newtheorem{rem}[thm]{\protect\remarkname}

\newtheorem{col}[thm]{\protect\corollaryname}

\newtheorem{ass}[thm]{\protect\assumptionname}

\providecommand{\definitionname}{Definition}
\providecommand{\propositionname}{Proposition}
\providecommand{\remarkname}{Remark}
\providecommand{\theoremname}{Theorem}
\providecommand{\lemmaname}{Lemma}
\providecommand{\corollaryname}{Corollary}
\providecommand{\conjecturename}{Conjecture}

\providecommand{\assumptionname}{Assumption}

\begin{document}

\title{Suboptimal Stabilizing Controllers for Linearly Solvable System}

\author{Yoke Peng Leong, Matanya B. Horowitz, and Joel W. Burdick 
\thanks{Y. P. Leong and M. B. Horowitz are with the Control and Dynamical Systems, California Institute of Technology, 
Pasadena, CA 91125, USA {\tt\small ypleong@caltech.edu}, {\tt\small mhorowit@caltech.edu}}%
\thanks{J. W. Burdick is with the Mechanical Engineering, California Institute of Technology, Pasadena, CA 91125, USA {\tt\small jwb@robotics.caltech.edu}}
}

\maketitle
\thispagestyle{empty}

\begin{abstract}

This paper presents a novel method to synthesize stochastic control Lyapunov functions for a class of nonlinear, stochastic control systems. In this work, the classical nonlinear Hamilton-Jacobi-Bellman partial differential equation is transformed into a linear partial differential equation for a class of systems with a particular constraint on the stochastic disturbance. It is shown that this linear partial differential equation can be relaxed to a linear differential inclusion, allowing for approximating polynomial solutions to be generated using sum of squares programming. It is shown that the resulting solutions are stochastic control Lyapunov functions with a number of compelling properties. In particular, a-priori bounds on trajectory suboptimality are shown for these approximate value functions. The result is a technique whereby approximate solutions may be computed with non-increasing error via a hierarchy of semidefinite optimization problems. 
\end{abstract}

\section{INTRODUCTION} \label{sec:intro}

The stabilization of nonlinear systems is a central problem in control engineering.  Lyapunov theory, wherein an energy-like function is used to show that some measure of distance from a stability point decays over time, is a critical tool for studying the convergence properties of a given system. 
Lyapunov theory may be generalized from analysis to synthesis of control systems using Control Lyapunov Function (CLF) \cite{freeman1996control}. However, the synthesis of a CLF for general systems
remains a challenging open question, due to the bilinearity between the Lyapunov function and control input in the Lyapunov equation.
%
%

A complementary and related domain in control engineering is the study of the Hamilton-Jacobi-Bellman (HJB) equation, a partial differential equaiton that governs the optimal control of a system. Methods to calculate the solution to the HJB equation via semidefinite programming have been proposed previously by Lasserre et al. \cite{lasserre2008nonlinear}. 
In this work, we propose an alternative line of study based on the linear structure of a particular form of the HJB equation. Since the late 1970s, researchers \cite{fleming1982logarithmic,holland1977new,dai1991stochastic,filliger2005relative} have made connections between stochastic optimal control and reaction-diffusion equation through a logarithmic transformation. This line of research has recently been the subject of focused study by Kappen \cite{Kappen:2005bn} and Todorov \cite{Todorov:2009wja}. 
These results have been developed in a number of compelling directions \cite{Stulp:2012jb, Dvijotham:2012tv, van2008graphical,rutquist2014hjb,shah2015dynamics}.

This paper combines these previously disparate fields of dynamic programming and Lyapunov theory by considering the value function, the solution to a stochastic HJB equation, as a Stochastic CLF (SCLF). The HJB solution is global, in that it incorporates all potential initial system states, and optimal. Here, we propose polynomial candidate approximate solutions to the HJB, extending recently developed tools in polynomial optimization to a new class of problems. It is already known that the solution to the deterministic HJB is in fact a CLF \cite{Primbs:1999tt}. This paper shows that our approximated value function solutions are SCLFs as well. 

A preliminary version of this work appeared in \cite{horowitz2014semidefinite} and \cite{horowitz2014admm}, where the use of semidefinite relaxations for solving the HJB were first considered. However, the stabilization properties of the resulting solutions were not investigated. Instead, these previous works focused on HJB solutions for path planning problems, and did not have guarantees on trajectory performance when using approximate solutions to the HJB. 


The rest of this paper is organized as follows. Section \ref{sec:background} reviews the linearly solvable HJB equations, control Lyapunov functions, and sum of squares programming.  Section \ref{sec:problem} introduces a relaxed formulation of the HJB solutions which is efficiently computable using the sum of squares methodology.  Section \ref{sec:analysis} analyzes the properties of the relaxed solutions, such as approximation errors relative to the exact solutions. This section also shows that the relaxed solutions are SCLFs, and that the resulting controller is stabilizing.  
An example is presented in Section \ref{sec:simulation} to illustrate the optimization technique and its performance. Section \ref{sec:conclusion} summarizes the findings of this work and discusses future research directions.

\section{BACKGROUND} \label{sec:background} 

This section briefly describes the notation and reviews necessary background on the linear HJB equation, SCLF, and SOS programming.

\subsection{Notation} \label{sec:notation}


\begin{table}[t]
	\centering
	\caption{Set notation}\label{tab:notation}
\begin{tabular}{p{0.06\textwidth}p{0.38\textwidth}}
	\hhline{==} 
	Notation & Definition \\
	\hline 
        $\mathbb{Z}_+$ & All positive integers \\
	$\mathbb{R}$ & All real numbers \\
	$\mathbb{R}_+$ & All nonnegative real numbers \\
	$\mathbb{R}^n$ & All $n$-dimensional real vectors \\
	$\mathbb{R}[x]$ & All real polynomial functions in $x$ \\
	$\mathbb{R}^{n \times m}$ & All $n \times m$ real matrices\\
	$\mathbb{R}^{n \times m}[x]$ & All $M \in \mathbb{R}^{n \times m}$ such that $M_{i,j} \in \mathbb{R}[x]~ \forall~i,j$\\
	$\mathcal{K}$ & All continuous nondecreasing functions $\mu: \mathbb{R}_+ \to \mathbb{R}_+$ such that  $\mu(0) = 0$, $\mu(r ) > 0$ if $r > 0$, and $\mu(r) \ge \mu(r')$ if $r > r'$ \\
	$\mathcal{C}^{k,k'}$ & All functions $f$ such that $f$ is $k$-differentiable with respect to the first argument and $k'$-differentiable with respect to the second argument \\
	\hhline{==}
\end{tabular}
\end{table}

Table \ref{tab:notation} summarizes the notation of different sets used in this work. A point on a trajectory, $x(t)\in\mathbb{R}^n$, at time $t$ is denoted $x_{t}$, while the segment of this trajectory over the interval $[t,T]$ is denoted by $x_{[t,T]}$.  

A compact domain in $\mathbb{R}^n$ is denoted as $\Omega$ where $\Omega \subset \mathbb{R}^n$, and its boundary is denoted as $\partial \Omega$. A domain $\Omega$ is a \emph{basic closed semialgebraic} set if there exists $g_i(x) \in \mathbb{R}[x]$ for $ i = 1,2,\ldots,m$ such that $\Omega = \{x \mid g_i(x) \ge 0 ~ \forall i = 1,2,\ldots,m\}$.  

Given a polynomial $p(x)$, $p(x)$ is positive on domain $\Omega$ if $p(x) >0 ~\forall x \in \Omega$, $p(x)$ is nonnegative on domain $\Omega$ if $p(x) \ge 0 ~\forall x \in \Omega$, and $p(x)$ is positive definite on domain $\Omega$ where $0 \in \Omega$, if $p(0) = 0$ and $p(x) >0$ for all $x \in \Omega \backslash \{0\}$.

If it exists, the infinity norm of a function is defined as $\norm{f}_\infty = \sup_x |f(x)|$ for $x\in\Omega$. To improve readability, a function, $f(x_1,\ldots,x_n)$, is abbreviated as $f$ when the arguments of the function are clear from the context.

\subsection{Linear Hamilton-Jacobi-Bellman (HJB) Equation}\label{sec:hjb}

Consider the following affine nonlinear dynamical system,
  \begin{equation} \label{eq:stochastic-dynamics}
    dx_t = \left(f(x_t) + G(x_t)u_t\right)dt + B(x_t) \, d\omega_t
  \end{equation}
where $x_t\in\Omega$ is the state at time $t$ in a compact state space domain $\Omega \subset \mathbb{R}^{n}$,  $u_t \in \mathbb{R}^{m}$ is the control input, $f(x)\in \mathbb{R}^n[x]$, $G(x)\in \mathbb{R}^{n \times m}[x]$,  $B(x)\in \mathbb{R}^{n \times l}[x] $ are real polynomial functions of the state variables $x$, and $\omega_{t} \in \mathbb{R}^{l}$ is a vector consisting of Brownian motions with covariance $\Sigma_\epsilon$, i.e., $\omega^i_t$ has independent increments with $\omega^i_{t}-\omega^i_{s}\sim\mathcal{N}(0,\Sigma_\epsilon(t-s))$, for $\mathcal{N}\left(\mu,\sigma^{2}\right)$ a normal distribution. The domain $\Omega$ is assumed to be a basic closed semialgebraic set defined as $\Omega = \{x \mid  g_i(x) \in \mathbb{R}[x], g_i(x) \ge 0 ~ \forall i = 1,2,\ldots,m\}$.  Without loss of generality, let $0 \in \Omega$ and $x=0$ be the equilibrium point, whereby $f(0) = 0$, $G(0) = 0$ and $B(0) = 0$. 

The goal is to minimize the following functional,
  \begin{gather} \label{eq:cost-functional}
       \mathbb{E}_{\omega_t}[J(x,u)]=\mathbb{E}_{\omega_t}\left[\phi(x_{T})+\int_{0}^{T}q(x_t)+\frac{1}{2}u_{t}^{T}Ru_{t} dt\right]
  \end{gather}
subject to (\ref{eq:stochastic-dynamics}), where $\phi \in \mathbb{R}[x]$, $\phi: \Omega \to \mathbb{R}_+$ represents a state-dependent terminal cost, $q \in \mathbb{R}[x]$, $q: \Omega \to \mathbb{R}_+ $ is state dependent cost, and $R \in \mathbb{R}^{m \times m}$ is a positive definite matrix. $T$, unknown a priori, is the time at which the system reaches the domain boundary or the origin. This problem is generally called the \emph{first exit} problem. The expectation $\mathbb{E}_{\omega_t}$ is taken over all realizations of the noise $\omega_t$. For stability of the resultant controller to the origin, $q$ and $\phi$ are also required to be positive definite functions. 
The solution to this minimization problem is known as the \emph{value function}, $V: \Omega \to \mathbb{R}_+$, where beginning from an initial point
$x_t$ at time $t$
  \begin{equation} \label{eq:value-def}
   V\left(x_t\right)=\min_{u_{[t,T]}}\mathbb{E}_{\omega_t}\left[J\left(x_{[t,T]},u_{[t,T]}\right)\right].
  \end{equation}

Based on dynamic programming arguments \cite[Ch.~III.7]{Fleming:2006tl}, the HJB equation
associated with this problem is a nonlinear, second order partial differential equation (PDE)
 \begin{multline}  \label{eq:hjb-pde-value}
     0 =  q+\left(\nabla_{x}V\right)^{T}f-\frac{1}{2}
           \left(\nabla_{x}V\right)^{T}GR^{-1}G^{T}\left(\nabla_{x}V\right)\\
                    +  \frac{1}{2}Tr\left(\left(\nabla_{xx}V\right)B\Sigma_{\epsilon}B^{T}\right)
  \end{multline}
with boundary condition $V(x) = \phi(x)$ and the optimal control effort takes the form
  \begin{equation}\label{eq:optimal-u}
        u^{*}=-R^{-1}G^{T}\nabla_x V.
  \end{equation}
For the stabilization problem on a compact domain, it is appropriate to set the boundary condition to be $\phi(x)=0$ for $x=0$, indicating zero cost accrued for achieving the origin, and $\phi(x)>0$ for $x \in \partial \Omega \setminus \{0\}$. In practice, $\phi(x)$ at the exterior boundary is usually chosen to be a large number depending on the applications to impose large penalty for exiting the predefined domain.

In general, \eqref{eq:hjb-pde-value} is difficult to solve due to its nonlinearity. However, with the assumption that there
exists a $\lambda>0$ and a control penalty cost $R$ in \eqref{eq:cost-functional} satisfying
  \begin{equation} \label{eq:noise-assumption}
    \lambda G(x_t)R^{-1}G(x_t)^{T}=B(x_t)\Sigma_{\epsilon}B(x_t)^{T}\triangleq \Sigma(x_t) \triangleq  \Sigma_{t},
   \end{equation}
and using the logarithmic transformation 
  \begin{equation} \label{eq:log-transform}
      V=-\lambda\log\Psi,
  \end{equation}
it is possible \cite{Todorov:2009wja,Kappen:2005bn}, after substitution and simplification, to obtain the following {\em linear}
PDE from (\ref{eq:hjb-pde-value}):
  \begin{gather} 
   0 = -\frac{1}{\lambda}q\Psi+f^{T}(\nabla_{x}\Psi)
      + \frac{1}{2}Tr\left(\left(\nabla_{xx}\Psi\right)\Sigma_{t}\right) \quad x \in \Omega \nonumber \\
   \Psi(x) = e^{- \frac{\phi(x)}{\lambda}}  \quad x \in \partial \Omega.
	 \label{eq:hjb-pde}
  \end{gather}
This transformation of the value function has been deemed the \emph{desirability} function \cite{Todorov:2009wja}. 
For brevity, define the following expression 
  \begin{equation*}\label{eq:L-psi}
     \mathcal{L}(\Psi)\triangleq f^{T}(\nabla_{x}\Psi)+\frac{1}{2}Tr\left(\left(\nabla_{xx}\Psi\right)\Sigma_{t}\right)
  \end{equation*}
and the function $\psi(x)$ at the boundary as  
\begin{equation*}
  \psi(x) \triangleq e^{- \frac{\phi(x)}{\lambda}} \quad x \in \partial \Omega.
\end{equation*}
%

Condition (\ref{eq:noise-assumption}) restricts the design of the control penalty $R$, such that control effort is highly penalized in subspaces with little noise, and lightly penalized in those with high noise. A specific case for which this condition is satisfied is for systems in which $B(x_t)=G(x_t)$. Additional discussion is given in \cite{Todorov:2009wja}.

\subsection{Stochastic Control Lyapunov Functions (SCLF)}\label{sec:clf}

Before the stochastic control Lyapunov function (SCLF) is introduced, the definitions for two forms of stability are provided, following the definitions in \cite[Ch. 5]{khasminskii2011stochastic}. 

\begin{defn} \label{def:stability_probability}
	Given \eqref{eq:stochastic-dynamics}, the equilibrium point at $x=0$ is stable in probability for $t\ge 0$ if for any $s\ge 0$ and $\epsilon > 0$,
	\begin{equation*}
		\lim_{x\to 0} P\left\{\sup_{t>s} |X^{x,s}(t)| > \epsilon \right\} =0
	\end{equation*}
	where $X^{x,s}$ is the trajectory of \eqref{eq:stochastic-dynamics} starting from $x$ at time $s$.
\end{defn}

Intuitively, Definition \ref{def:stability_probability} is similar to the notion of stability for deterministic systems. 
The following is a stronger stability definition that is similar to the notion of asymptotic stability for deterministic systems.

\begin{defn} \label{def:asymp_stability_probability}
	Given \eqref{eq:stochastic-dynamics}, the equilibrium point at $x=0$ is asymptotically stable in probability if it is stable in probability and 
	\begin{equation*}
		\lim_{x\to 0} P\left\{\lim_{t \to \infty} |X^{x,s}(t)| = 0 \right\} =1
	\end{equation*}
	where $X^{x,s}$ is the trajectory of \eqref{eq:stochastic-dynamics} starting from $x$ at time $s$.
\end{defn}

For stochastic systems, the SCLF and Lyapunov theorems are defined as follows. 
\begin{defn}  \label{def:stochastic_clf}
A stochastic control Lyapunov function (SCLF) for system (\ref{eq:stochastic-dynamics}) is a positive
definite function $\mathcal{V} \in \mathcal{C}^{2,1}$ on a compact domain $\mathcal{O} =   \Omega \cup \{0\}  \times \{t>0\}$ such that 
\begin{gather*}
	\mathcal{V}(0,t) = 0, \quad \mathcal{V}(x,t)\ge \mu(|x|) \quad \forall ~t \nonumber\\
	\exists~ u(x,t) \mbox{ s.t. } L(\mathcal{V}(x,t)) \leq 0  \quad \forall~(x,t) \in \mathcal{O} \backslash \{(0, t)\} \label{eq:sclf}
\end{gather*}
where $\mu \in \mathcal{K}$, and
    \begin{multline}
		L(\mathcal{V}) =\partial_t \mathcal{V}+ \nabla_x \mathcal{V}^T (f + Gu)+ \frac{1}{2}  Tr((\nabla_{xx} \mathcal{V}) B \Sigma_\epsilon B^T). \label{eq:LV}
    \end{multline}
\end{defn}

\begin{thm}\label{thm:stochastic_clf} \cite[Thm. 5.3]{khasminskii2011stochastic}
	For system \eqref{eq:stochastic-dynamics}, assume that there exists a SCLF and a $u$ defined in Definition \ref{def:stochastic_clf}. Then, the equilibrium point $x = 0$ is stable in probability, and $u$ is a stabilizing controller. 
\end{thm}

To achieve the stronger condition of asymptotic stability in probability, we have the following result.
\begin{thm}\label{thm:stochastic_clf_asymp}\cite[Thm. 5.5 and Cor. 5.1]{khasminskii2011stochastic}
	For system \eqref{eq:stochastic-dynamics}, suppose that in addition to the existence of a SCLF and a $u$ defined in Definition \ref{def:stochastic_clf}, $u$ is time-invariant,
	\begin{gather*}
		\mathcal{V}(x,t) \le \mu'(|x|) \quad \forall ~t \\ 
		\quad L(\mathcal{V}(x,t)) < 0 \quad \forall~(x,t) \in \mathcal{O} \backslash \{(0, t)\}
	\end{gather*}
	where $\mu' \in \mathcal{K}$.
	Then, the equilibrium point $x = 0$ is asymptotically stable in probability, and $u$ is an asymptotically stabilizing controller.
\end{thm}

\subsection{Sum of Squares (SOS) Programming}

This section provides a brief review of SOS programming, the tool by which we will use to generate approximate solutions to the HJB equation.
A complete introduction to the subject of SOS programming is available in \cite{Parrilo:2003fh}.
\begin{defn} \label{def:SOS}
A multivariate polynomial $f(x)$ is a sum of squares (SOS) if there
exist polynomials $f_{0}(x),\ldots,f_{m}(x)$ such that
  \[ f(x)=\sum_{i=0}^{m}f_{i}^{2}(x). \]
  The set of SOS polynomials in $x$ is denoted as $\mathbb{S}[x]$.
\end{defn}
A sufficient condition for non-negativity of a polynomial $f(x)$ is that $f(x)\in \mathbb{S}[x]$. This seemingly simple fact is compelling, as testing the membership of a polynomial in $\mathbb{S}[x]$ may be performed as a convex problem \cite{Parrilo:2003fh}.
%

\begin{thm}\cite[Thm. 3.3]{Parrilo:2003fh} \label{thm:sos-test-sdp}
	The existence of a SOS decomposition of a polynomial in $n$
	variables of degree $2d$ can be decided by solving a semidefinite programming (SDP) feasibility
	problem. 
\end{thm}

Hence, by adding SOS constraints to the set of all positive polynomials, testing nonnegativity of a polynomial becomes a tractable SDP problem. The converse question, is a nonnegative polynomial necessarily a SOS, is unfortunately false, indicating that this test is conservative \cite{Parrilo:2003fh}. Nonetheless, SOS feasibility is sufficiently powerful for our purposes.
 
Theorem \ref{thm:sos-test-sdp} guarantees a tractable procedure to determine whether a particular polynomial, possibly parameterized, is a SOS polynomial. Our method combines multiple polynomial constraints into an optimization formulation. 
To do so, we need to define the following polynomial set.
\begin{defn}
	The preordering of polynomials $g_i(x) \in \mathbb{R}[x]$ for $i = 1,2,\ldots, m$ is the set 
	\begin{multline}
		P(g_1,\ldots,g_m) \\= \left\{\left. \sum_{\nu \in \{0,1\}^m} s_\nu(x)g_1(x)^{\nu_1} \cdots g_m(x)^{\nu_m}  \right\arrowvert s_{\nu} \in \mathbb{S}[x]\right\}.
	\end{multline}
\end{defn}

The following proposition is useful to incorporate the domain $\Omega$ in our optimization formulation later.
\begin{prop}\label{lem:positive}
	Given $f(x) \in \mathbb{R}[x]$, if $f(x) \in P(g_1,\ldots,g_m)$, on the domain $\Omega = \{x \mid  g_i(x) \in \mathbb{R}[x], g_i(x) \ge 0, i \in \{ 1,2,\ldots,m\}\}$, then $f(x)$ is nonnegative on $\Omega$. If there exists another polynomial $f'(x)$ such that $f'(x) \ge f(x)$, then $f'(x)$ is also nonnegative on $\Omega$.
\end{prop}
To illustrate how this proposition applies, consider a polynomial $f(x)$ on a domain defined by $x \in [-1,1]$. The bounded domain can be equivalently defined by polynomials $g_1(x) = 1+x$ and $g_2(x) = 1-x$. To certify that $f(x) \ge 0$ on the specified domain, construct a function $h(x) = s_1(x) (1+x) + s_2(x) (1-x) + s_3(x) (1+x)(1-x)$ where $s_i \in \mathbb{S}[x]$ and certify that $f(x) - h(x) \ge 0$. Notice that $h(x) \in P(1+x,1-x)$, so $h(x) \ge 0$. If $f(x) - h(x) \ge 0$, then $f(x) \ge h(x) \ge 0$. Proposition \ref{lem:positive} is applied here. Finding the correct $s_i(x)$ is not trivial in general. Nonetheless, as mentioned earlier, if we further impose that $f(x) - h(x) \in \mathbb{S}[x]$, then checking if there exists $s_i(x)$ such that $f(x) - h(x) \in \mathbb{S}[x]$ becomes a semidefinite feasibility program as given by Theorem \ref{thm:sos-test-sdp}. More concretely, the procedure may begin with a limited polynomial degree for $s_i(x)$, increasing the degree until a certificate is found (if one exists) or the computation resources are exhausted. 

To simplify notation in later text, given a domain $\Omega = \{x \mid  g_i(x) \in \mathbb{R}[x], g_i(x) \ge 0, i \in \{ 1,2,\ldots,m\}\}$, we set the notation $P(\Omega) = P(g_1,\ldots,g_m)$. 

\section{Sum-of-Squares Relaxation of the HJB PDE}\label{sec:problem}

This section demonstrates how SOS programming can be used to solve the linear HJB via an SOS relaxation. We would like to emphasize the following standing assumption, typical of moment and SOS-based methods \cite{lasserre2008nonlinear,Parrilo:2003fh}.
\begin{ass} Assume that system \eqref{eq:stochastic-dynamics} evolves on a compact domain $\Omega \subset \mathbb{R}^n$ that is also a basic closed semialgebraic set such that $\Omega = \{x \mid g_i(x) \in \mathbb{R}[x], g_i(x) \ge 0, i \in \{1, \ldots, k\}\}$ for some $k \ge 1$. Then, the boundary $\partial\Omega$ is polynomial representable. We use the notation $\partial\Omega = \{x \mid h_i(x) \in \mathbb{R}[x], \prod_{i=1}^m h_i(x)=0\}$ for some $m \ge 1$ to describe this boundary. 
\end{ass}

The following definitions formalize several operators that will be useful in later text.
\begin{defn}\label{def:domain}
	Given a basic closed semialgebraic set $\Omega = \{x \mid g_i(x) \in \mathbb{R}[x], g_i(x) \ge 0, i \in \{1, \ldots, k\}\}$ and a set of SOS polynomials,
$\mathcal{S} = \{s_\nu(x)\mid s_\nu(x) \in\mathbb{S}[x], \nu \in \{0,1\}^k\}$, 
	define the operator $\mathcal{D}$ as 
	\begin{equation*}
		\mathcal{D}(\Omega,\mathcal{S}) = \sum_{\nu \in \{0,1\}^k} s_\nu(x) g_1(x)^{\nu_1}\cdots g_k(x)^{\nu_k}
	\end{equation*}
	where $\mathcal{D}(\Omega,\mathcal{S}) \in P(\Omega)$.
\end{defn}

\begin{defn}\label{def:boundary}
	Given a polynomial inequality, $p(x) \ge 0$, the boundary of a compact set $\partial\Omega = \{x \mid h_i(x) \in \mathbb{R}[x], \prod_{i=1}^m h_i(x)=0\}$ and a set of  polynomials, $\mathcal{T} = \{t_i(x) \mid t_i(x) \in\mathbb{R}[x], i \in \{1, \ldots, m\}\}$,  
	define the operator $\mathcal{B}$ as 
	\begin{equation*}
		\mathcal{B}(p(x),\partial \Omega,\mathcal{T}) = \{p(x)- t_{i}(x) h_{i}(x) \mid  i \in \{1, \ldots, m\}\}
	\end{equation*}
	where $\mathcal{B}$ returns a set of polynomials that is nonnegative on $\partial \Omega$.
\end{defn}

\subsection{Relaxation of the HJB equation}\label{sec:hjb-relax}
For the remainder of this paper, we assume that the unique solution to \eqref{eq:hjb-pde-value} and \eqref{eq:hjb-pde} exists in the viscosity solutions sense (see \cite{Fleming:2006tl}, Chapter V) and denote the unique solutions as $V^*$ and $\Psi^*$ respectively.

The equality constraints of \eqref{eq:hjb-pde} may be relaxed (in either direction) as follows
  \begin{gather}
    \frac{1}{\lambda}q\Psi - \mathcal{L}(\Psi) \le (\ge) 0 \nonumber \\
		\Psi(x) \le (\ge) \psi(x) \qquad  x \in \partial \Omega. \label{eq:over-approximation-relax}
  \end{gather}

\noindent This relaxation provides a point-wise bound to the true solution, and it may be enforced via SOS programming. In particular, a solution to \eqref{eq:over-approximation-relax}, denoted as $\Psi_l$ ($\Psi_u$), is a lower (upper) bound on the solution $\Psi^*$ over the domain $\Omega$.

\begin{prop} \label{prop:Psi_bound}
Given a smooth function $ \Psi_l $ ($\Psi_u$) that satisfies  \eqref{eq:over-approximation-relax}, then $\Psi_l$ ($\Psi_u$) is a viscosity subsolution (supersolution) and $\Psi_l \le \Psi^{*}$ ($\Psi_u \ge \Psi^{*}$) for all $x \in \Omega$.
\end{prop}
\begin{proof}
By \cite[Def. 2.2]{crandall1992user}, the solution $\Psi_l$ is a viscosity subsolution. Note that $\Psi^*$ is both a viscosity subsolution and a viscosity supersolution, and $\Psi_l \le \Psi^{*}$ on the boundary $\partial \Omega$. Hence, by the maximum principle for viscosity solutions \cite[Thm 3.3]{crandall1992user}, $\Psi_l \le \Psi^{*}$ for all $x \in \Omega$. Similar argument applies for $\Psi_u$.
\end{proof}


Because the logarithmic transform \eqref{eq:log-transform} is monotonic, one can relate
these bounds on the desirability function to bounds on the value function as follows:
\begin{prop} \label{prop:v-upper-lower}
If the solution to \eqref{eq:hjb-pde-value} is $V^*$, given solutions $V_u = -{\lambda}\log \Psi_l$ and $V_l = -{\lambda}\log \Psi_u$ from \eqref{eq:over-approximation-relax}, then $V_u \ge V^*$ and $V_l \le V^*$.
\end{prop}


\subsection{Controller Synthesis}\label{sec:oc_synthesis}

Given that relaxation \eqref{eq:over-approximation-relax} results in a point-wise upper and lower bound to the exact solution of \eqref{eq:hjb-pde}, we construct the following optimization that provides a suboptimal controller with bounded residual error:
  \begin{alignat}{2}
	\min_{\Psi_l,\Psi_u} \quad& \epsilon \label{eq:hjbjoin} \\
	s.t. \quad &  \frac{1}{\lambda} q \Psi_l 
              - \mathcal{L}(\Psi_l) \leq 0  && \quad x \in\Omega\nonumber \\
 	& 0 \leq  \frac{1}{\lambda} q \Psi_u 
              - \mathcal{L}(\Psi_u) &&  \quad x \in \Omega \nonumber \\
	& \Psi_u - \Psi_l \leq \epsilon &&\quad x\in \Omega \nonumber \\
	& 0 \le \Psi_l  \le\psi  \le \Psi_u   && \quad x \in\partial \Omega \nonumber \\
	& \partial_{x^i} \Psi_l \le 0 && \quad x^i \ge 0 \nonumber \\
	& \partial_{x^i} \Psi_l \ge 0 && \quad x^i \le 0 \nonumber \\
	& \Psi_l(0) = 1  \nonumber 
  \end{alignat} 
  where $x^i$ is the $i$-th component of $x \in \Omega$. As mentioned in Section \ref{sec:hjb-relax}, the first two constraints result from the relaxations of the HJB equation, and the fourth constraint
arises from the relaxation of the boundary conditions. The third constraint ensures that the
solution error is bounded by $\epsilon$, and the last three constraints ensure that the solution yields a stabilizing controller, as will be made clear in Section \ref{sec:analysis}.

In order to solve \eqref{eq:hjbjoin} as a semidefinite optimization problem, we restrict the polynomial inequalities such that they are SOS polynomials instead of nonnegative polynomials. Therefore, after applying Proposition \ref{lem:positive} to the domain constraints, the resulting optimization is 
  \begin{alignat}{2}
	\min_{\Psi_l,\Psi_u,\mathcal{S},\mathcal{T}} \quad& \epsilon \label{eq:hjbjoin-sos} \\
	s.t. \quad &  \frac{1}{\lambda} q \Psi_l 
              + \mathcal{L}(\Psi_l)- \mathcal{D}(\Omega,\mathcal{S}_1) \in \mathbb{S}[x]\nonumber \\
 	&  \frac{1}{\lambda} q \Psi_u 
              - \mathcal{L}(\Psi_u) - \mathcal{D}(\Omega,\mathcal{S}_2) \in \mathbb{S}[x] \nonumber \\
	& \epsilon - (\Psi_u - \Psi_l ) - \mathcal{D}(\Omega,\mathcal{S}_3)  \in \mathbb{S}[x]\nonumber \\
	&  \mathcal{B}( \Psi_l ,\partial\Omega,\mathcal{T}_1) \in \mathbb{S}[x]\nonumber \\
	&  \mathcal{B}(\psi -\Psi_l ,\partial\Omega,\mathcal{T}_2)\in \mathbb{S}[x] \nonumber \\
	&  \mathcal{B}( \Psi_u - \psi,\partial\Omega,\mathcal{T}_3) \in \mathbb{S}[x]\nonumber \\
	& -\partial_{x^i} \Psi_l - \mathcal{D}(\Omega \cap \{x^i \ge 0\}, \mathcal{S}_4) \in \mathbb{S}[x]\nonumber \\
	& \partial_{x^i} \Psi_l  - \mathcal{D}(\Omega \cap \{-x^i \ge 0\}, \mathcal{S}_5) \in \mathbb{S}[x]\nonumber \\
	& \Psi_l(0) = 1  \nonumber 
  \end{alignat}
  where  $\mathcal{S} = (\mathcal{S}_1,\ldots,\mathcal{S}_4,\mathcal{S}_5)$, $\mathcal{S}_i \subseteq \mathbb{S}[x]$ is defined as in Definition \ref{def:domain}, $\mathcal{T} = (\mathcal{T}_1,\mathcal{T}_2,\mathcal{T}_3)$, and $\mathcal{T}_j \subseteq \mathbb{R}[x]$ is defined as in Definition \ref{def:boundary}. With a slight abuse of notation, $\mathcal{B}(\cdot) \in \mathbb{S}[x]$ implies that each polynomial in $\mathcal{B}(\cdot)$ is a SOS polynomial.

If the degrees of polynomials are fixed, optimization \eqref{eq:hjbjoin-sos} is convex and may be solved as an SDP via Theorem \ref{thm:sos-test-sdp}. The next section will discuss the systematic approach we used to solve the optimization. 
  
\begin{rem}
By definition, the viscosity solution is a continuous function \cite[Def. 2.2]{crandall1992user}. Consequently, the solution $\Psi^*$ is a continuous function defined on a bounded domain. Hence, $\Psi_u$ and $\Psi_l$ can be made arbitrary close to $\Psi^*$ by the Stone-Weierstrass Theorem \cite{rudin1964principles} in \eqref{eq:hjbjoin}. However, this guarantee is lost when $\Psi_u$ and $\Psi_l$ are restricted to be SOS polynomials. The feasible set of the optimization problem \eqref{eq:hjbjoin-sos} is therefore not necessarily non-empty for a given polynomial degree. 
\end{rem}

\subsection{Hierarchy of SOS programs}

%

Let $d$ be the maximum degree of $\Psi_l$, $\Psi_u$ and polynomials in $\mathcal{S}$ and $\mathcal{T}$, and denote $(\Psi^d_u,\Psi^d_l, \mathcal{S}^d,\mathcal{T}^d, \epsilon^d)$ as a solution to \eqref{eq:hjbjoin-sos} when the maximum polynomial degree is fixed at $d$. The hierarchy of SOS programs with increasing polynomial degree produces a sequence of possibly empty solutions $(\Psi^d_u,\Psi^d_l, \mathcal{S}^d,\mathcal{T}^d, \epsilon^d)_{d\in I}$, where $I \subset \mathbb{Z}_+$. This sequence will be shown in the next section to improve, under the metric of the objective in \eqref{eq:hjbjoin-sos}. 
The use of such hierarchies has become common in polynomial optimization \cite{lasserre2001global, Parrilo:2003fh}. Once a satisfactory error is achieved or computational resources run out, the lower bound $\Psi_l$ is used to compute the suboptimal controller.  The suboptimal controller $u^\epsilon$ for a given error $\epsilon$ is computed as $u^\epsilon = -R^{-1}G^T \nabla_x V_u$ where $V_u = - {\lambda} \log {\Psi_l}$. The next section will analyze the properties of the solutions and the suboptimal controller.

\section{ANALYSIS}\label{sec:analysis}

This section establishes appealing properties of the solutions to the optimization \eqref{eq:hjbjoin-sos}
that are relevant for feedback control. First, we show that the solutions in the SOS program hierarchy are uniformly bounded relative to the exact solutions. We next prove that the solutions to the relaxed stochastic HJB equation are SCLFs, and they yield stabilizing controllers. Finally, we show that the costs of using the approximate solutions as controllers are bounded above by the approximated value functions. 

\subsection{Properties of the Approximated Desirability Functions}

%
First, compute the approximation error of the true desirability function $\Psi_l$ or $\Psi_u$ obtained from optimization \eqref{eq:hjbjoin-sos}.
\begin{prop} \label{prop:psi_estimate}
Given a solution $(\Psi_u,\Psi_l, \mathcal{S},\mathcal{T},\epsilon)$ to \eqref{eq:hjbjoin-sos} for a fixed degree $d$, the approximation
error of a desirability function is bounded as $||\Psi - \Psi^*||_{\infty} \leq \epsilon$ where
$\Psi$ is either $\Psi_u$ or $\Psi_l$.
\end{prop}

\begin{proof}
By Corollary \ref{prop:Psi_bound}, $\Psi_l$ is the lower bound of $\Psi^*$,
and $\Psi_u$ is the upper bound of $\Psi^*$. So, $\epsilon \geq \Psi_u - \Psi_l \geq 0$ and $\Psi_u
\geq \Psi^* \geq \Psi_l$. Combining both inequalities, one has $\Psi_u - \Psi^* \leq \epsilon$
and $\Psi^*- \Psi_l \leq \epsilon$. Therefore, $||\Psi - \Psi^*||_{\infty} \leq \epsilon$ where
$\Psi$ is either $\Psi_u$ or $\Psi_l$.
\end{proof}

\begin{prop} \label{prop:gap}
The hierarchy of SOS programs consisting of solutions to \eqref{eq:hjbjoin-sos} with
increasing polynomial degree produces a sequence of solutions $(\Psi^d_u,\Psi^d_l, \mathcal{S}^d,\mathcal{T}^d, \epsilon^d)$ such that $\epsilon^{d+1} \le \epsilon^{d}$ for all $d$.
\end{prop}

\begin{proof}
	Polynomials of degree $d$ form a subset of polynomials of degree $d+1$. Thus, at a higher polynomial degree $d+1$, a previous solution at a lower polynomial degree $d$ is still a feasible solution when the coefficients for monomials with total degree $d+1$ is set to 0. Consequently, the optimal value $\epsilon^{d+1}$ cannot be smaller than $\epsilon^{d}$ for all $d$. 
\end{proof}

Although the \emph{bound} on the pointwise error is non-increasing, the actual error may in fact increase between iterations. We bound this variation as follows.
\begin{col} \label{col:gap}
Suppose $||\Psi^{d} - \Psi^*||_{\infty} \le \epsilon^{d}$ and $||\Psi^{d+1} - \Psi^*||_{\infty} = \gamma^{d+1}$. Then, $\gamma^{d+1}\le \epsilon^{d}$. 
\end{col}
\begin{proof}
From Proposition \ref{prop:gap}, $\gamma^{d+1}\le \epsilon^{d+1} \le \epsilon^{d}$. 
\end{proof}

Note that $\epsilon$ is only non-increasing as
polynomial degree increases. Therefore, Proposition \ref{prop:gap} and Corollary \ref{col:gap} does not guarantee a convergence of $\epsilon$ to zero. 

\subsection{Properties of the Approximated Value Function}\label{sec:cost_convergence}

We now investigate the implications of Corollary \ref{col:gap} upon the value function. 
Henceforth, denote the solution to \eqref{eq:hjb-pde-value} as $V^*(x_t) = \min_{u[t:T]}\mathbb{E}_{\omega_t}[J(x_t)] = - \lambda \log \Psi^*(x_t)$, and the suboptimal value function computed from the solution of \eqref{eq:hjbjoin-sos} as $V_u = - {\lambda} \log {\Psi_l}$. 

\begin{thm} \label{thm:cost-upper}
$V_u$ is an upper bound of the optimal cost $V^*$ such that
  \begin{equation}
	 0 \leq V_u - V^* \leq -\lambda \log\left(1-\min\left\{1,\frac{\epsilon}{\eta}\right\}\right) 
  \end{equation}
where $\eta = e^{-\frac{\norm{V^*}_\infty}{\lambda}}$.
\end{thm}
\begin{proof}
By Proposition \ref{prop:v-upper-lower}, $V_u \ge V^* $ and hence, $V_u - V^* \ge 0$. To prove the other inequality, by Proposition \ref{prop:psi_estimate},
  \begin{align*}
	V_u - V^*  &= - \lambda \log \frac{\Psi_l}{\Psi^*} \leq -\lambda \log \frac{\Psi^*-\epsilon}{\Psi^*} \leq -\lambda \log \left(1-\frac{\epsilon}{\eta}\right) .
  \end{align*}	
The last inequality holds because $\Psi^* \geq e^{-\frac{\norm{V^*}_\infty}{\lambda}}$ by definition
in \eqref{eq:log-transform}. Since $\Psi_l$ is the lower bound of $\Psi^*$, the right hand side of the
first equality is always a positive number. Therefore, $V_u$ is a point-wise upper bound of
$V^*$.
\end{proof}

\begin{col} \label{col:cost-upper}
Let $V_u^{d} = -\lambda \log \Psi^{d}_l$ and $V_u^{d+1} = -\lambda \log \Psi^{d+1}_l$. If $V_u^d - V^*\le \epsilon^d$ and $V_u^{d+1} - V^* = \gamma^{d+1}$, then $\gamma^{d+1} \le -\lambda \log\left(1-\min\left\{1,\frac{\epsilon^{d}}{\eta}\right\}\right)$.
\end{col}


At this point, we have shown that the lower bound of the desirability function gives an upper
bound of the suboptimal cost. More importantly, the upper bound of the suboptimal cost is non-increasing as the polynomial degree increases. 

\subsection{The Exact and Approximate HJB solutions are SCLFs}

%
%
%
Here, we show that the approximate value function derived from the lower desirability approximation,
$\Psi_l$, is a SCLF.

\begin{thm}\label{thm:lower-sclf}
$V_u$ is a stochastic control Lyapunov function according to Definition \ref{def:stochastic_clf}. 
\end{thm}

\begin{proof}
The constraint $\Psi_l(0) = 1$ ensures that $V_u(0) = - \lambda \log {\Psi_l(0)} = 0$. Notice that all terms in $J(x,u)$ from \eqref{eq:cost-functional} are positive definite, resulting in $V^*$ being a positive definite function. In addition, by Proposition \ref{prop:v-upper-lower}, $V^u \ge V^*$. Hence, $V^u$ is also a positive definite function. The second and third to last constraints in \eqref{eq:hjbjoin-sos} ensures that $\Psi_l$ is nonincreasing. Hence, $V_u$ is nondecreasing satisfying $\mu(|x|) \le V_u(x) \le \mu'(|x|)$ for some $\mu, \mu' \in \mathcal{K}$.
	
	 Next, show that there exists a $u$ such that $L(V_u) \leq 0$. Following (\ref{eq:optimal-u}), let 
	  \begin{equation} \label{eq:suboptimal}
	      u^\epsilon = -R^{-1}G^T\nabla_x V_u\ .
	  \end{equation}
	  Notice that from the definition of $V_u$, $\nabla_x V_u = - \frac{\lambda}{\Psi_l} \nabla_x \Psi_l$ and $\nabla_{xx} V_u =
\frac{\lambda}{\Psi_l^2}(\nabla_x \Psi_l) (\nabla_x \Psi_l)^T - \frac{\lambda}{\Psi_l}
\nabla_{xx}\Psi_l$. So, $u = \frac{\lambda}{\Psi_l} R^{-1}G^T \nabla_x \Psi_l$.
Then, from \eqref{eq:LV},
  \begin{align*}
     L(V_u) &= -\frac{\lambda}{\Psi_l}(\nabla_x \Psi_l)^T
            (f + \frac{\lambda}{\Psi_l} G R^{-1}G^T \nabla_x \Psi_l)\\
          & + \frac{1}{2} Tr\left(\left( \frac{\lambda}{\Psi_l^2}(\nabla_x \Psi_l )
            (\nabla_x \Psi_l)^T - \frac{\lambda}{\Psi_l} \nabla_{xx}\Psi_l \right)B \Sigma_\epsilon B\right)
  \end{align*}
  where $\partial_t V_u = 0$ because $V_u$ is not a function of time. 
Applying the assumption in \eqref{eq:noise-assumption} and simplifying,
  \begin{multline*}
      L(V_u)  = -\frac{\lambda}{\Psi_l}(\nabla_x \Psi_l)^T f 
	    -\frac{\lambda}{2 \Psi_l^2} (\nabla_{x}\Psi_l)^T \Sigma_t \nabla_{x}\Psi_l \\
                   -\frac{\lambda}{2\Psi_l} Tr\left(\left(\nabla_{xx}\Psi_l \right) \Sigma_t \right).
  \end{multline*}
From the first constraint in \eqref{eq:hjbjoin-sos},
  \begin{gather*}
     \frac{1}{\lambda} q \Psi_l - f^{T}(\nabla_{x}\Psi_l)
           -\frac{1}{2}Tr\left(\left(\nabla_{xx}\Psi_l\right)\Sigma_{t}\right) \leq 0 \implies\\
    - \frac{\lambda}{\Psi_l}(\nabla_{x}\Psi_l)^T f \leq -  
       q  + \frac{\lambda}{2 \Psi_l}Tr\left(\left(\nabla_{xx}\Psi_l\right)\Sigma_{t}\right). 
  \end{gather*}
Substituting this inequality into $L(V_u)$ and simplifying yields
  \begin{align}
	L(V_u) \leq -q -\frac{\lambda}{2 \Psi_l^2} (\nabla_{x}\Psi_l)^T \Sigma_t \nabla_{x}\Psi_l \leq 0 \label{eq:lvle0}
  \end{align}
because $q\ge 0$, $\lambda > 0$ and $\Sigma_t$ is positive semidefinite by definition. 
Since $V_u$ satisfies Definition \ref{def:stochastic_clf}, $V_u$ is a SCLF.	
\end{proof}

\begin{col} \label{col:local-stable}
The suboptimal controller $u^\epsilon = -R^{-1}G^T\nabla_x V_u$ is stabilizing in probability within the domain $\Omega$. If $\Sigma_t$ is a positive definite matrix, the suboptimal controller $u^\epsilon = -R^{-1}G^T\nabla_x V_u\ $ is asymptotically stabilizing in probability within the domain $\Omega$.
\end{col}
\begin{proof}
This corollary is a direct consequence of the constructive proof of Theorem \ref{thm:lower-sclf} and Theorem \ref{thm:stochastic_clf}.
\end{proof}

\setcounter{figure}{1}
\begin{figure*}[b]
      \centering
	  \subfigure[ ]{\includegraphics[trim =12mm 65mm 15mm 70mm, clip,width=0.32\textwidth]{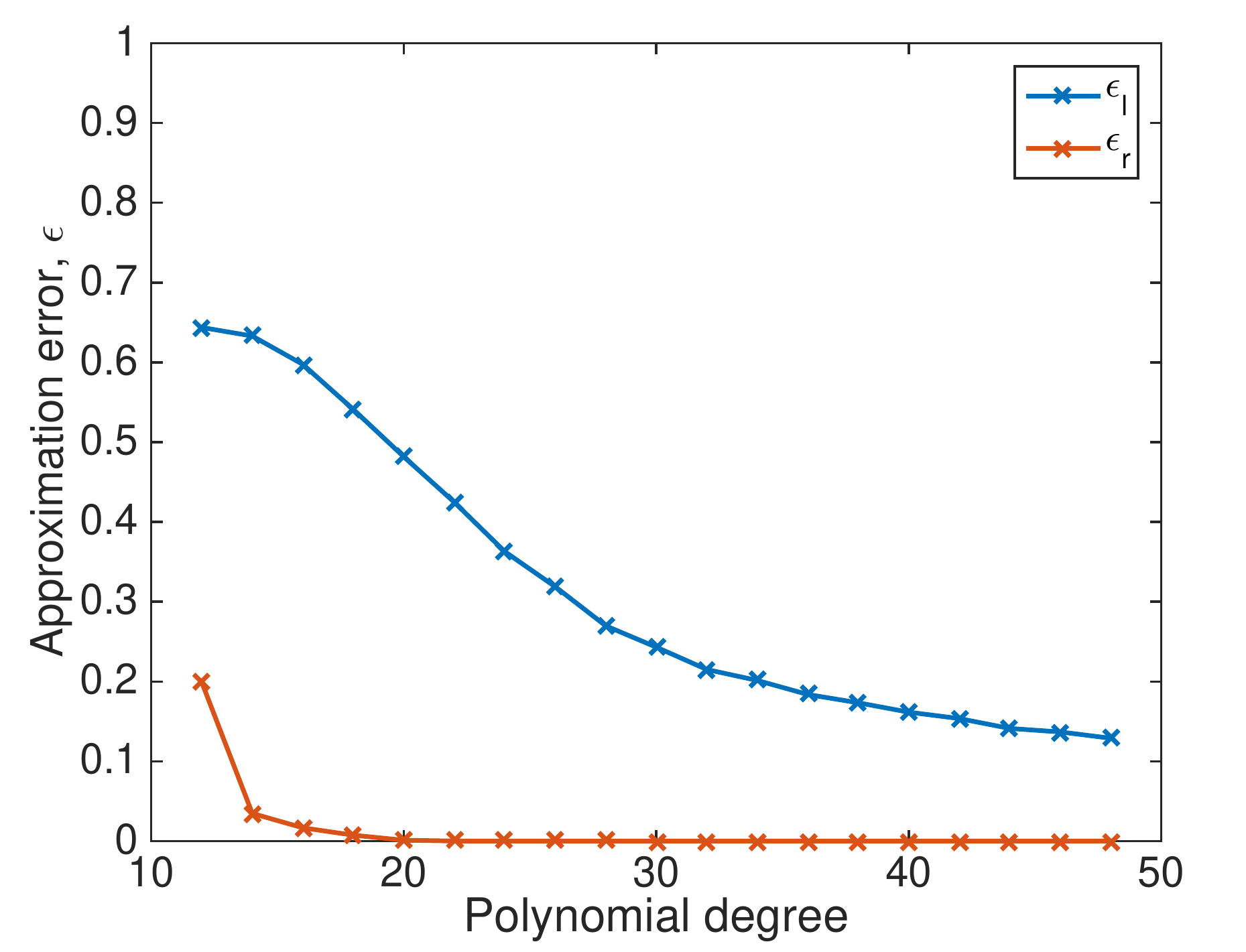}\label{fig:ex-1-solution-epsilon}  }  
	  \subfigure[ ]{\includegraphics[trim =12mm 65mm 15mm 70mm, clip,width=0.32\textwidth]{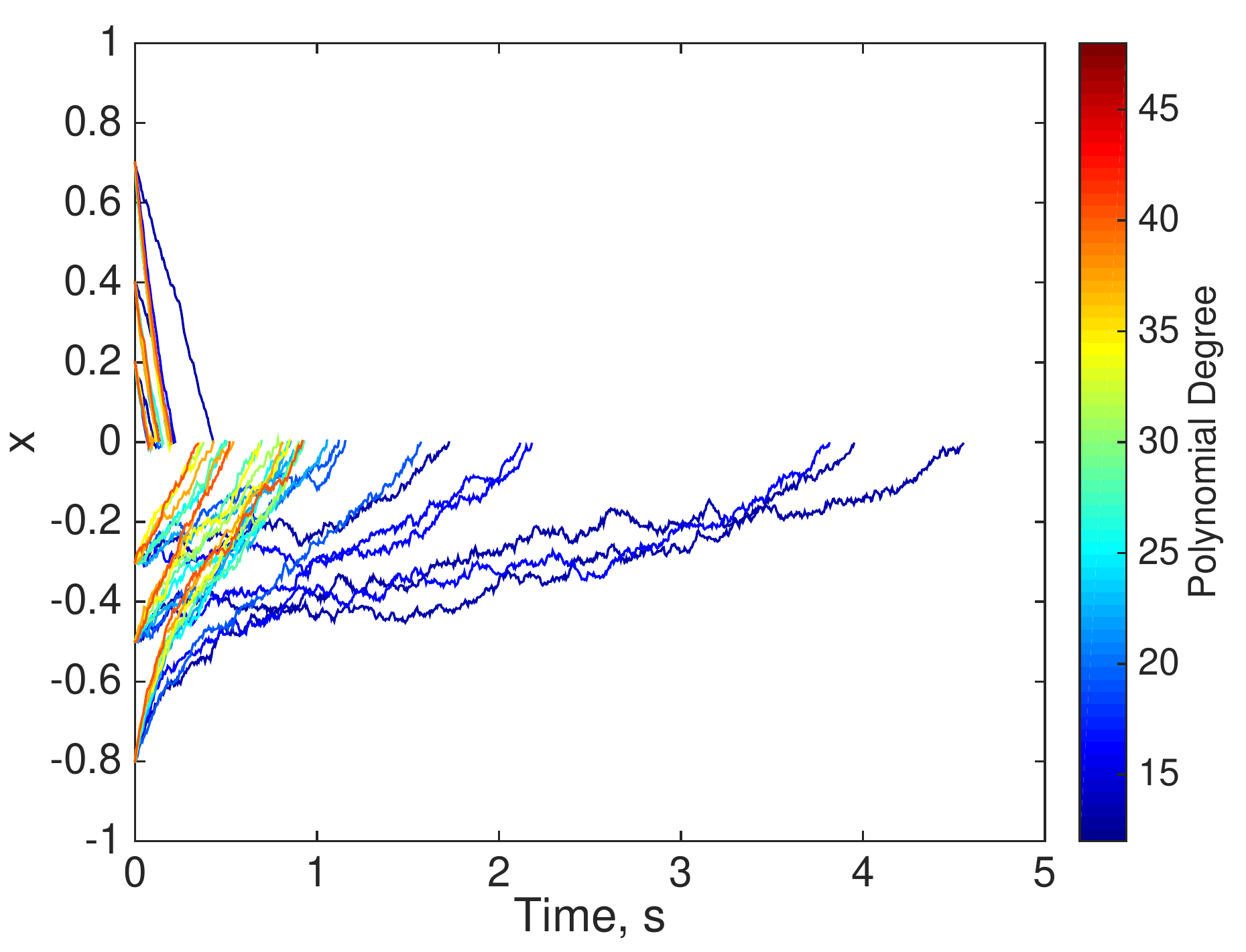} \label{fig:ex-1-solution-traj}}  
	  \subfigure[ ]{\includegraphics[trim =15mm 65mm 15mm 70mm, clip,width=0.32\textwidth]{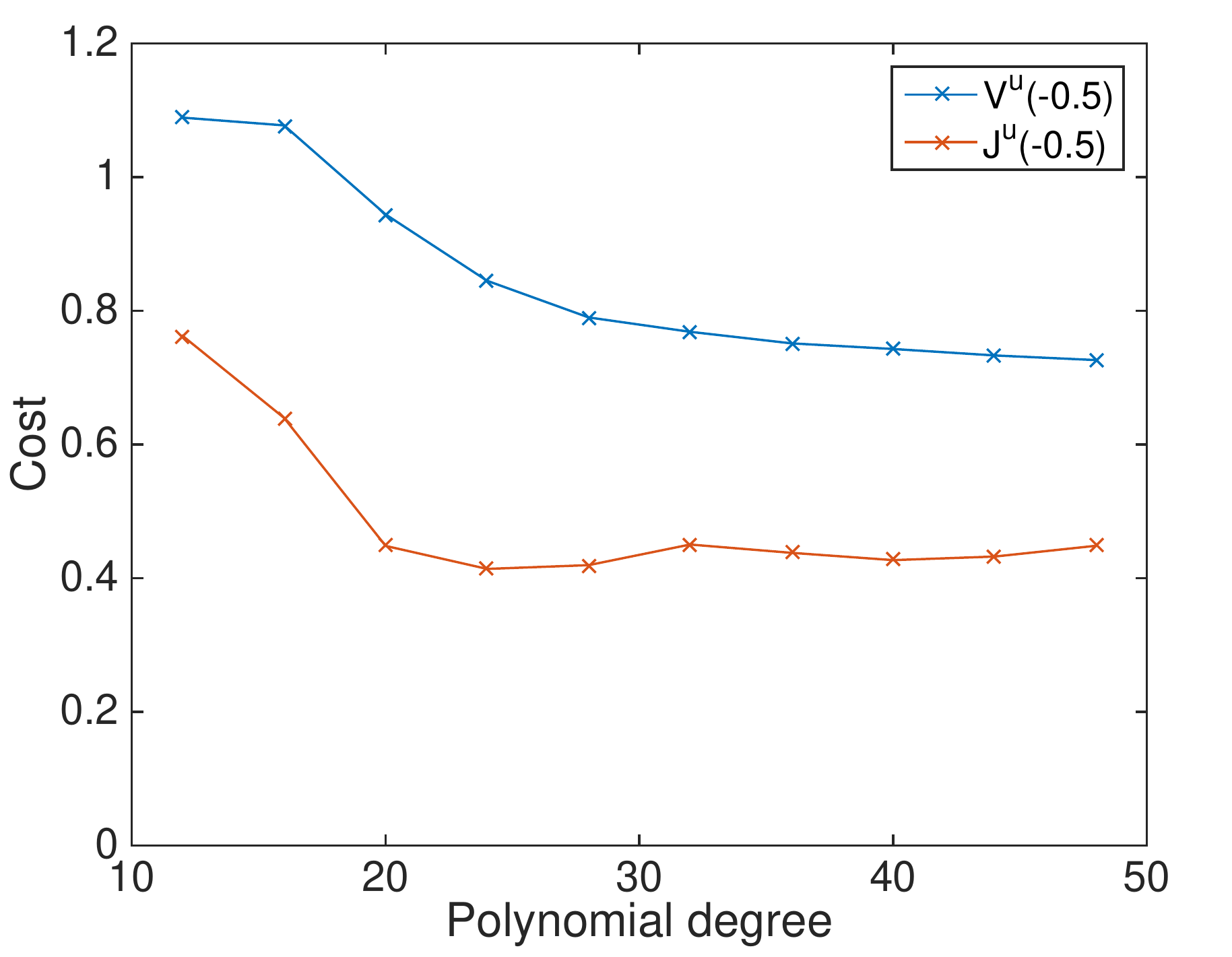}\label{fig:ex-1-solution-cost}} 	  
      \vspace{-.2cm}
      \caption{Computational results of system \eqref{eq:1dexample}. (a) Convergence of the objective function of  \eqref{eq:hjbjoin-sos} as the degree of polynomial increases. The approximation error for $x \le 0$ is denoted as $\epsilon_l$ and the approximation error for $x \ge 0$ is denoted as $\epsilon_r$. (b) Sample trajectories using controller computed from optimization problem \eqref{eq:hjbjoin-sos} with different polynomial degrees starting from six randomly chosen initial points. (c) The comparison between $J_u$ and $V_u$ for different polynomial degrees whereby $J_u$ is the expected cost and $V_u$ is the value function computed from optimization problem \eqref{eq:hjbjoin-sos}. The initial condition is fixed at $x_0 = -0.5$.}     
 \end{figure*}

\subsection{Bound on the Total Trajectory Cost} 

We conclude this section by showing that the expected total trajectory cost incurred by the system while 
operating under the suboptimal controller of (\ref{eq:suboptimal}) is bounded.

\begin{thm}\label{thm:total-traj-cost}
Given the control law $u^\epsilon=-R^{-1}G^T \nabla_x V_u$,
  \begin{equation}
    J_u \le V_u\leq V^* -\lambda \log\left(1-\min\left\{1,\frac{\epsilon}{\eta}\right\}\right)
  \end{equation}
where $J_u = \mathbb{E}_{\omega_t} [\phi_{T}(x_{T})+\int_{0}^{T}r(x_{t},u^\epsilon_{t}) dt]$, the expected cost of
the system when using the given control law, $u^\epsilon$.
\end{thm}
\begin{proof}
By It\^{o}'s formula,
	\begin{gather*}
		dV_u(x_t) = L(V_u)(x_t) dt + \nabla_x V_u(x_t) B(x_t) d\omega_t. 	
	\end{gather*}
where $L(V)$ is defined in \eqref{eq:LV}.
Then,
	\begin{align}
		V_u(x_t)  = V_u(x_0,0) + &\int^t_0 L(V_u)(x_s) ds \nonumber\\
		& + \int^t_0 \nabla_xV_u(x_s) B(x_s) d\omega_s. \label{eq:temp}
	\end{align}
Take the expectation of this equation to get
	\begin{align*}
		\mathbb{E}_{\omega_t}[V_u(x_t)] &= V_u(x_0,0) +\mathbb{E}_{\omega_t}\left[\int^t_0 L(V_u)(x_s) ds \right]
	\end{align*}
whereby the last term of (\ref{eq:temp}) drops out because the noise is assumed to have zero mean. The expectations of the other terms return the same terms because they are deterministic.
From \eqref{eq:lvle0}, 
\begin{align*}
	L(V_u) &\le  -q -\frac{\lambda}{2 \Psi_l^2} (\nabla_{x}\Psi_l)^T \Sigma_t \nabla_{x}\Psi_l \\
	&=-q -\frac{1}{2} \left(\nabla_{x}V_u\right)^{T}GR^{-1}G^{T}\left(\nabla_{x}V_u\right)\\
	&= -q -\frac{1}{2} (u^\epsilon)^T R u^\epsilon
\end{align*}
where the first equality is given by the logarithmic transformation and the second equality is given by the control law $u^\epsilon=-R^{-1}G^T \nabla_x V_u$. 
	Therefore, 
	\begin{align*}
		\mathbb{E}_{\omega_t}[V_u(x_t)] &= V_u(x_0) +\mathbb{E}_{\omega_t}\left[\int^t_0 L(V_u)(x_s) ~ds\right] \\
		&\le V_u(x_0) -  \mathbb{E}_{\omega_t}\left[\int^t_0  q(x_s) +\frac{1}{2} (u^\epsilon_s)^T R u^\epsilon_s ~ds\right]  \\
		&= V_u(x_0) - J(x_0,u^\epsilon)+\mathbb{E}_{\omega_t}[\phi(x_T)]
	\end{align*}
Therefore, $V_u(x_0) - J(x_0,u^\epsilon) \ge \mathbb{E}_{\omega_t}[V_u(x_t)-\phi(x_t)]$. By definition, $V_u(x_T) \geq
\phi(x_T)$ for all $x_T \in \Omega$. Thus, $\mathbb{E}_{\omega_t}[V_u(x_T)-\phi(x_T)] \geq 0$. Consequently, $V_u(x_0) - J(x_0,u^\epsilon) \geq 0$, and $V_u(x_0) \geq J(x_0,u^\epsilon)$. Lastly, Theorem
\ref{thm:cost-upper} gives the second inequality in the theorem.
\end{proof}

\section{Numeric Examples}\label{sec:simulation}

This section studies the computational characteristics of our method using a scalar unstable system. The optimization parser
YALMIP \cite{yalmip} was used in conjunction with the semidefinite optimization package MOSEK \cite{andersen2000mosek} to solve the optimization problem \eqref{eq:hjbjoin-sos}. 

Consider the following unstable scalar nonlinear system
   \begin{equation} 
       dx=\left(-x^{3}+5x^{2}+3x+u\right)dt+d\omega \label{eq:1dexample}
   \end{equation}
 on the domain $x\in \Omega = \{x \mid -1 \le x \le 1 \}$. The noise model considered is Gaussian white noise with zero mean and variance $\Sigma_\epsilon=1$. The goal is to stabilize the system at the origin. Instead of zero, we choose the boundary at two ends of the domain to be  $\Psi(-1)=20e^{-10}$ and $\Psi(1)=20e^{-10}$. At the origin, the boundary is set as $\Psi(0) = 1$. We set $q=x^2$, and $R=1$. Because of the natural division of the domain, the solutions for both domains can be represented by smooth polynomials respectively, and solved independently.
 
 \setcounter{figure}{0}
 \begin{figure}[t]
    \centering
    \includegraphics[trim = 12mm 65mm 15mm 70mm, clip, width=0.45\textwidth]{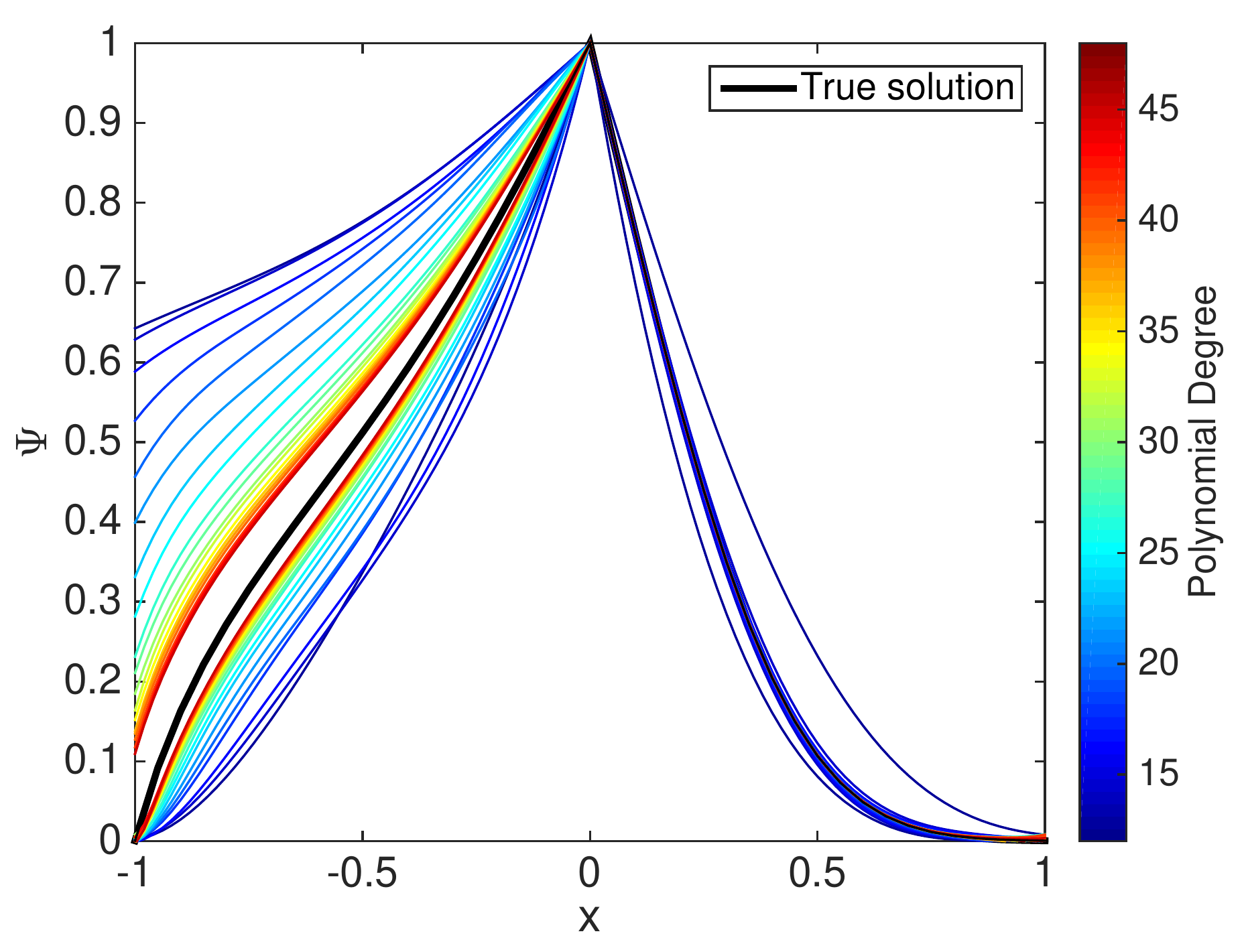}
    \caption{The desirability function for varying polynomial degree. The true solution is the black curve.}
       \label{fig:Desirability-ex-1}
  \end{figure}

 %
 %
 %

The desirability functions that results from solving \eqref{eq:hjbjoin-sos} for varying polynomial degrees are shown in Figure \ref{fig:Desirability-ex-1}. The optimization problem is not feasible for polynomial degree below 12. The true solution is computed using Mathematica. The kink at the origin is expected because the HJB PDE solution is not necessarily smooth at the boundary, and in this situation the origin is itself a boundary between the two domain halves. The approximation error $\epsilon$ for both partitions is shown in Figure \ref{fig:ex-1-solution-epsilon} for increasing polynomial degree. As seen in the plots, the approximation improves as the polynomial degree increases. 
 
To quantify the performance of the controller, a Monte Carlo experiment is performed. For each polynomial degree that is feasible, the controller obtained from $\Psi_l$ in optimization \eqref{eq:hjbjoin-sos} is implemented in 20 simulations of the system subject to random samples of Gaussian white noise with $\Sigma_\epsilon = 1$. The initial condition is fixed at $x_0 =-0.5$ and $t=0$. The continuous system is integrated numerically using Euler integration with step size of 0.005s. The simulation is terminated when the trajectories enter the interval $[-0.005,0.005]$ centered on the origin. Figure \ref{fig:ex-1-solution-cost} shows the comparison between $J_u(x_0,t)$ and $V_u(x_0,t)$ for different polynomial degrees whereby $J_u$ is the expected cost and $V_u$ is the value function computed from $\Psi_l$ in optimization \eqref{eq:hjbjoin-sos}. Figure \ref{fig:ex-1-solution-traj} illustrates several sample trajectories. In general, the trajectories converge earlier when the polynomial degree is higher. This observation is expected because the approximation error is smaller as the polynomial degree increases.

\addtolength{\textheight}{-10cm} 
\section{CONCLUSION}\label{sec:conclusion}

This paper proposes a novel method to solve the linear Hamilton Jacobi Equation of an optimal control problem with nonlinear, stochastic systems dynamics via sum of squares programming. Analytical results provide guarantees on the suboptimality of trajectories when using the
approximate solutions for controller design. Consequently, one can synthesize a suboptimal stabilizing controller to nonlinear, stochastic dynamical systems.

To improve the algorithm, the monomials of the polynomial approximation can be chosen strategically in order to decrease computation time while achieving high accuracy. Thus, a promising future direction is the synthesis of the work presented here with that of \cite{horowitz2014linear}, where
HJB equations were solved in dimension twelve and higher. 
To improve the numerical conditioning of these optimization
techniques, other numerical schemes are also under investigation \cite{horowitz2014admm}. 

There remains the question of the limitations placed by the structural constraint
\eqref{eq:noise-assumption}. A compelling research question is the suboptimality of controllers and
trajectories when approximating systems that do not adhere to the constraint, such as deterministic systems or those with noise in states without a control channel. 


\addtolength{\textheight}{-10cm}   
\bibliographystyle{IEEEtran}
\bibliography{references}

\begin{thebibliography}{10}
\providecommand{\url}[1]{#1}
\csname url@samestyle\endcsname
\providecommand{\newblock}{\relax}
\providecommand{\bibinfo}[2]{#2}
\providecommand{\BIBentrySTDinterwordspacing}{\spaceskip=0pt\relax}
\providecommand{\BIBentryALTinterwordstretchfactor}{4}
\providecommand{\BIBentryALTinterwordspacing}{\spaceskip=\fontdimen2\font plus
\BIBentryALTinterwordstretchfactor\fontdimen3\font minus
  \fontdimen4\font\relax}
\providecommand{\BIBforeignlanguage}[2]{{%
\expandafter\ifx\csname l@#1\endcsname\relax
\typeout{** WARNING: IEEEtran.bst: No hyphenation pattern has been}%
\typeout{** loaded for the language `#1'. Using the pattern for}%
\typeout{** the default language instead.}%
\else
\language=\csname l@#1\endcsname
\fi
#2}}
\providecommand{\BIBdecl}{\relax}
\BIBdecl

\bibitem{freeman1996control}
R.~A. Freeman and J.~A. Primbs, ``Control {L}yapunov functions: new ideas from
  an old source,'' in \emph{Proceedings of the 35th IEEE Conference on Decision
  and Control}, vol.~4, 1996, pp. 3926--3931.

\bibitem{lasserre2008nonlinear}
J.~B. Lasserre, D.~Henrion, C.~Prieur, and E.~Tr{\'e}lat, ``Nonlinear optimal
  control via occupation measures and {L}{M}{I}-relaxations,'' \emph{SIAM
  Journal on Control and Optimization}, vol.~47, no.~4, pp. 1643--1666, 2008.

\bibitem{fleming1982logarithmic}
W.~H. Fleming, \emph{Logarithmic transformations and stochastic control}.\hskip
  1em plus 0.5em minus 0.4em\relax Springer, 1982.

\bibitem{holland1977new}
C.~J. Holland, ``A new energy characterization of the smallest eigenvalue of
  the schr{\"o}dinger equation,'' \emph{Communications on Pure and Applied
  Mathematics}, vol.~30, no.~6, pp. 755--765, 1977.

\bibitem{dai1991stochastic}
P.~Dai~Pra, ``A stochastic control approach to reciprocal diffusion
  processes,'' \emph{Applied mathematics and Optimization}, vol.~23, no.~1, pp.
  313--329, 1991.

\bibitem{filliger2005relative}
R.~Filliger and M.-O. Hongler, ``Relative entropy and efficiency measure for
  diffusion-mediated transport processes,'' \emph{Journal of Physics A:
  Mathematical and General}, vol.~38, no.~6, p. 1247, 2005.

\bibitem{Kappen:2005bn}
H.~Kappen, ``Linear theory for control of nonlinear stochastic systems,''
  \emph{Physical Review Letters}, vol.~95, no.~20, 2005.

\bibitem{Todorov:2009wja}
E.~Todorov, ``{Efficient computation of optimal actions},'' \emph{Proceedings
  of the National Academy of Sciences (PNAS)}, vol. 106, no.~28, pp.
  11\,478--11\,483, 2009.

\bibitem{Stulp:2012jb}
F.~Stulp, E.~A. Theodorou, and S.~Schaal, ``{Reinforcement Learning With
  Sequences of Motion Primitives for Robust Manipulation},'' \emph{{IEEE}
  Transactions on Robotics}, vol.~28, no.~6, pp. 1360--1370, 2012.

\bibitem{Dvijotham:2012tv}
K.~Dvijotham and E.~Todorov, ``Linearly-solvable optimal control,''
  \emph{Reinforcement learning and approximate dynamic programming for feedback
  control}, pp. 119--141, 2012.

\bibitem{van2008graphical}
B.~Van Den~Broek, W.~Wiegerinck, and B.~Kappen, ``Graphical model inference in
  optimal control of stochastic multi-agent systems.'' \emph{J. Artif. Intell.
  Res. (JAIR)}, vol.~32, pp. 95--122, 2008.

\bibitem{rutquist2014hjb}
P.~Rutquist, T.~Wik, and C.~Breitholtz, ``{Solving the Hamilton-Jacobi-Bellman
  equation for a stochastic system with state constraints},'' in \emph{IEEE
  53rd Conference on Decision and Control (CDC)}, 2014, pp. 1840--1845.

\bibitem{shah2015dynamics}
S.~K. Shah and H.~G. Tanner, ``Dynamics-compatible potential fields using
  stochastic perturbations,'' in \emph{23rd Mediterranean Conference on Control
  and Automation (MED)}, 2015, pp. 278--283.

\bibitem{Primbs:1999tt}
J.~A. Primbs, V.~Nevisti{\'c}, and J.~C. Doyle, ``{Nonlinear optimal control: A
  control Lyapunov function and receding horizon perspective},'' \emph{Asian
  Journal of Control}, vol.~1, no.~1, pp. 14--24, 1999.

\bibitem{horowitz2014semidefinite}
M.~Horowitz and J.~Burdick, ``Semidefinite relaxations for stochastic optimal
  control policies,'' in \emph{American Controls Conf. (ACC)}, 2014, pp.
  3006--3012.

\bibitem{horowitz2014admm}
M.~B. Horowitz, I.~Papusha, and J.~W. Burdick, ``Domain decomposition for
  stochastic optimal control,'' in \emph{53rd Conference on Decision and
  Control (CDC)}, 2014, pp. 1866--1873.

\bibitem{Fleming:2006tl}
W.~H. Fleming and H.~M. Soner, \emph{{Controlled Markov processes and viscosity
  solutions}}.\hskip 1em plus 0.5em minus 0.4em\relax New York: Springer, 2006,
  vol.~25.

\bibitem{khasminskii2011stochastic}
R.~Khasminskii, \emph{Stochastic Stability of Differential Equations}.\hskip
  1em plus 0.5em minus 0.4em\relax Springer Science \& Business Media, 2011,
  vol.~66.

\bibitem{Parrilo:2003fh}
P.~A. Parrilo, ``{Semidefinite programming relaxations for semialgebraic
  problems},'' \emph{Mathematical Programming}, vol.~96, no.~2, pp. 293--320,
  2003.

\bibitem{crandall1992user}
M.~G. Crandall, H.~Ishii, and P.-L. Lions, ``User’s guide to viscosity
  solutions of second order partial differential equations,'' \emph{Bulletin of
  the American Mathematical Society}, vol.~27, no.~1, pp. 1--67, 1992.

\bibitem{rudin1964principles}
W.~Rudin, \emph{Principles of Mathematical Analysis}.\hskip 1em plus 0.5em
  minus 0.4em\relax McGraw-Hill, New York, 1964, vol.~3.

\bibitem{lasserre2001global}
J.~B. Lasserre, ``Global optimization with polynomials and the problem of
  moments,'' \emph{SIAM Journal on Optimization}, vol.~11, no.~3, pp. 796--817,
  2001.

\bibitem{yalmip}
J.~Lofberg, ``{YALMIP : a toolbox for modeling and optimization in MATLAB},''
  in \emph{IEEE International Symposium on Computer Aided Control Systems
  Design}, 2004, pp. 284--289.

\bibitem{andersen2000mosek}
E.~D. Andersen and K.~D. Andersen, ``The {MOSEK} interior point optimizer for
  linear programming: an implementation of the homogeneous algorithm,'' in
  \emph{High performance optimization}.\hskip 1em plus 0.5em minus 0.4em\relax
  Springer, 2000, pp. 197--232.

\bibitem{horowitz2014linear}
M.~B. Horowitz, A.~Damle, and J.~W. Burdick, ``{Linear Hamilton Jacobi Bellman
  Equations in High Dimensions},'' in \emph{IEEE 53rd Conference on Decision
  and Control (CDC)}, 2014, pp. 5880--5887.

\end{thebibliography}

\vfill

\end{document}